\documentclass{amsart}
\usepackage[margin=3.5cm]{geometry}
\usepackage[dvipsnames]{xcolor}

\usepackage{amsfonts, amsthm, amssymb,verbatim,amscd,amsmath, blindtext}

\usepackage{multirow,multicol}
\usepackage{graphicx}
\graphicspath{ {./images/} }
\usepackage{enumitem,yhmath}
\usepackage{amssymb}
\usepackage{ytableau}
\usepackage{tikz}
\usepackage{tikz-cd}
\usepackage{float, pifont}
\usepackage{mathtools}
\usepackage{float}
\usepackage[pagebackref,colorlinks=true,citecolor=blue,linkcolor=BrickRed]{hyperref}
\usepackage{caption} 
\newtheorem{theorem}{Theorem}[section]
\newtheorem{lemma}[theorem]{Lemma}
\newtheorem{corollary}[theorem]{Corollary}
\newtheorem{proposition}[theorem]{Proposition}

\newtheorem{observation}[theorem]{Observation}
\newtheorem{claim}[theorem]{Claim}
\newtheorem{notation}[theorem]{Notation}

\theoremstyle{remark}
\newtheorem{remark}[theorem]{Remark}

\theoremstyle{definition}

\DeclareMathOperator{\reg}{reg}
\DeclareMathOperator{\lcm}{lcm}
\DeclareMathOperator{\mingens}{Mingens}
\DeclareMathOperator{\lex}{lex}
\DeclareMathOperator{\diam}{diam}
\DeclareMathOperator{\trim}{trim}
\DeclareMathOperator{\length}{length}

\title{Trees whose path ideals have linear quotients}
\author[Chau]{Trung Chau}
\address{Chennai Mathematical Institute, India}
\email{chauchitrung1996@gmail.com}

\author[Das]{Kanoy Kumar Das}
\address{Chennai Mathematical Institute, India}
\email{ kanoydas0296@gmail.com; kanoydas@cmi.ac.in}

\author[Dutta Dhar]{Animikha Dutta Dhar}
\address{Chennai Mathematical Institute, India}
\email{animikha.ug2024@cmi.ac.in}

\author[Karanth]{Pranath S Karanth}
\address{Chennai Mathematical Institute, India}
\email{pranathk.ug2024@cmi.ac.in}

\author[Suswaram]{Aniruda Suswaram}
\address{Chennai Mathematical Institute, India}
\email{anirudas.ug2024@cmi.ac.in; anirudasuswaram06@gmail.com}


\begin{document}
	
    \begin{abstract}
    For any integer $n$, we classify all trees whose $n$-path ideals have linear quotients. 
    \end{abstract}

    \maketitle

	
	\section{Introduction}
	A monomial ideal $I\subseteq \Bbbk[x_1,x_2,\dots , x_m]$ is said to have \emph{linear quotients} if there is an ordering of the monomials in $\mingens(I)=\{f_1, f_2, \dots , f_s\}$ such that the successive colon ideals $(f_1,f_2, \dots , f_i):f_{i+1}$ are generated by variables for all $1\leq i\leq s-1$. Ideals with linear quotients were introduced by Herzog and Takayama \cite{HerzogTakayama2002} and they have strong combinatorial implications. In the case of square-free monomial ideals, linear quotients and shellable simplicial complexes are dual concepts. This fact attracted many researchers to identify ideals with linear quotients and explore their properties \cite{ JahanZheng2010,SharifanVarbaro2008}. 

    In \cite{Villarreal1990}, Villarreal introduced the class of edge ideals of graphs and studied their various algebraic properties which relate to the combinatorics of the corresponding graphs. Arguably the most celebrated result in this theme is that of Fr\"oberg \cite{Froberg}, which states that a graph is co-chordal if and only if its edge ideal has linear resolution. Since then, an astounding amount of research has been carried out around edge ideals, as this class connects two different areas of mathematics, namely, graph theory and commutative algebra. Later on, several generalizations of edge ideals, such as path ideals, connected ideals, weighted oriented edge ideals, etc., have been introduced to carry forward this theme of research. A natural question one can ask is whether one can find an analog of the celebrated result of Fr\"oberg \cite{Froberg} for these classes of ideals, that is, to classify ideals which have linear resolution. In this context, having linear quotients implies having linear resolution, and in fact, the converse holds for edge ideals.

    
    Conca and De Negri \cite{ConcaDeNegri1998} were the first to consider path ideals of directed graphs to study $M$-sequences. Over the years, many algebraic properties and invariants of path ideals of directed graphs such as (Castelnuovo-Mumford) regularity, projective dimension, graded betti numbers, Cohen-Macaulayness, etc., have been intensively investigated \cite{BouchatBrown2017,  BouchatHaOkeefe2011,Erey2020, KianiMadani2016}. In this work, we consider path ideals of simple undirected graphs. Let $G$ be a simple graph with $V(G)=\{x_1, x_2, \dots , x_m\}$. The \textit{$n$-path ideal of $G$}, denoted by $J_n(G)$, is defined to be the ideal of $\Bbbk[x_1,x_2, \dots , x_m]$ given by \[
    J_n(G)\coloneqq \left(\prod_{j=1}^n x_{i_j}\mid x_{i_1},x_{i_2},\dots ,x_{i_n} \text{ forms a }n\text{-path in }G\right).
    \] 
    
    Alilooee and Faridi \cite{AlilooeeFaridi2018, AlilooeeFaridi2015} considered $n$-path ideals of lines and cycles, and studied their Betti numbers. 
    In \cite{Banerjee2017}, it has been shown that if $G$ is gap-free and claw free, then $J_n(G)$ has linear resolution for $n=3,4,5,6$; and in addition, if $G$ is whiskered $K_4$-free, then $J_n(G)$ has linear resolution for all $n\geq 3$. In \cite{DasRoySaha2024}, $3$-path ideals of chordal graphs have been studied, and an exact formula for the regularity of $n$-path ideals of caterpillar graphs was given. Other related works in this direction are \cite{HangVu2025, KumarSarkar2024}. In these work, the authors use different techniques to determine the regularity of $J_n(G)$, obtaining instances of linear resolution as a corollary. In other words, their results do not imply that these ideals have linear quotients.
    
    Recently, in \cite{AJM2024}, the authors considered connected ideals of graphs, which is another generalization of edge ideals, and as a consequence of their main result \cite[Theorem 5.1]{AJM2024}, it follows that $J_3(G)$ have linear quotients if and only if it has linear resolution if and only if $G$ does not contain $P_3+P_3$, the disjoint union of two paths of length 2, as an induced subgraph, when $G$ is a tree. Moreover, $J_2(G)$ is the edge ideal of $G$, and thus a corollary of Fr\"oberg theorem is: for a tree $G$, the ideal $J_2(G)$ has linear quotients if and only if it has linear resolution if and only if $G$ does not contain $P_2+P_2$, the disjoint union of two edges, as an induced subgraph. In this article, for $n\geq 4$, we consider $n$-path ideals of trees and investigate when they have linear quotients (and linear resolution). The goal is to extend the aforementioned results for $J_2(G)$ and $J_3(G)$.
    
    Any square-free monomial ideal can be treated as an edge ideal of a hypergraph. 
    A necessary condition for a square-free monomial ideal to have linear resolution is that the induced matching number of the corresponding hypergraph to be one. This fact gives one of the obvious forbidden structures in case of path ideals of graphs: if $J_n(G)$ has linear resolution, then $G$ does not contain the disjoint union of two paths of length $n$, namely $P_n+P_n$, as an induced subgraph. However, for general $n$, this is not the only forbidden structure, even for trees.
     We show that, in the case of trees, there is a family of graphs, which we denote as $L_{n,k}$ where $n\geq 5, k\in [3,n-2]$, are also forbidden to be an induced subgraph of $G$ if $J_n(G)$ has linear resolution. Indeed, to be more precise, we prove the following:

    \begin{theorem}[{Theorem~\ref{thm:main-section-5}}]\label{thm:main}
        Let $G$ be a tree and $n\geq 4$ a positive integer. Then the following statements are equivalent:
        \begin{enumerate}
            \item $J_n(G)$ has linear resolution;
            \item $J_n(G)$ has linear quotients;
            \item either of the following holds:
            \begin{itemize}
                \item[(i)] $n=4$, and  $G$ does not contain $P_4+P_4$ or $L_{5,3}$ as an induced subgraph;
                \item[(ii)] $n\geq 5$, and $G$ does not contain $P_n+P_n$ or $L_{n,k}$ as an induced subgraph for  any $k\in [3,\ (n+1)/2].$
            \end{itemize}
        \end{enumerate}
    \end{theorem}

   Some consider path ideals more complicated than other generalizations of edge ideals, e.g., connected ideals. One reason for this is the difficulty of determining $J_n(G)$ for a given $n$ and given $G$. For connected ideals (or edge ideals in particular), more edges in the graph almost always equal more generators for the ideals, a principle that does not hold for path ideals. For example, $J_4(G)=(0)$ for any star graph $G$ (a graph whose edges share a common vertex). The novelty in our approach is that for any $n\geq 4$ and any tree $G$ such that $J_n(G)$ has linear resolution, we can determine an induced subgraph $H$ of $G$ such that $J_n(G)=J_n(H)$ and $J_n(H)$ can be easily computed. We call the operation to obtain $H$ from $G$ \emph{trimming}, which we shall explore in Section~\ref{sec:trimmed trees}. As a preview, let $G$ be the following tree.
   
   	\begin{figure}[!hbt]
		
		\begin{center}
			\begin{tikzpicture}[every node/.style={circle,draw,inner sep=2pt}, every label/.append style={shape= rectangle, minimum height=5mm, align=center}]
				\foreach \i in {1,2,3,4,5,6,7} {
					\node[label={90:$x_{\i}$}] (x\i) at (\i, 0) {};
				}
				\node[label={90:$x_{8}$}] (x8) at (5,-1) {};
                \node[label={90:$x_{9}$}] (x9) at (6,-1) {};

				\draw (x1) -- (x2) -- (x3) -- (x4) -- (x5) -- (x6) -- (x7);
				\draw (x4) -- (x8) -- (x9);
			\end{tikzpicture}
		\end{center}
		\label{fig:intro}
		\caption{The graph $G$}
	\end{figure}
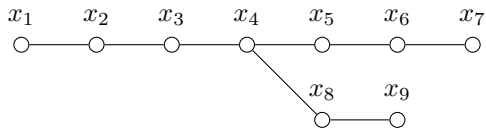

    Let $H$ be the induced subgraph of $G$ with the vertex set $\{x_1,x_2,\dots ,x_8\}$. It can be verified straightforwardly that $J_7(H)=J_7(G)=(x_1x_2x_3x_4x_5x_6x_7)$. In fact, this is the smallest example where the trimming operation results in a proper subgraph of $G$. In other words, the effectiveness of this method can only be observed when one considers $J_n(G)$ when $n\geq 7$.

    This paper is organised as follows. In Section~\ref{sec:preliminaries}, we recall preliminary notions from commutative algebra and graph theory, and prove a few auxiliary results. In Section~\ref{sec:forbidden} we list possible forbidden structures in a tree. Section~\ref{sec:trimmed trees} discusses the trimming operation, and shows that in the absence of the forbidden structures, the trimming operation depends only on the original graph. In Section~\ref{sec:main}, we prove our main result.

    \section*{Acknowledgements}
    Chau and Das are partially supported by a grant from the Infosys Foundation.

	\section{Preliminaries}\label{sec:preliminaries}
    In this section, we recall several notions from graph theory and commutative algebra, and establish a few auxiliary results that are essential in the later parts of this article.
    
	\subsection{Graph terminology}
	
	Let $G=(V(G),E(G))$ be a finite simple graph. A graph $H=(V(H),E(H))$ is called an \emph{induced subgraph} of $G$ if $V(H)\subseteq V(G)$ and $E(H)=\{xy \colon x,y\in V(H) \text{ and }xy\in E(G) \}$. In particular, an induced subgraph of $G$ is uniquely determined by its vertex set. A sequence of vertices $\mathcal{P}\colon x_1,\dots, x_n$ in $G$ is called a \emph{path} from $x_1$ to $x_n$ if these are $n$ distinct vertices and $\{x_i,x_{i+1}\}\in E(G)$ for any $i\in [n-1]$. In this case, $x_1$ and $x_n$ are called two \emph{ends} of $\mathcal{P}$. The \emph{length} of $\mathcal{P}$, denoted by $\length(\mathcal{P})$, is defined to be one less than the number of vertices it contains. For any two graphs $G$ and $H$ that share no common vertex, let $G+H$ denote the \emph{disjoint union} of $G$ and $H$, i.e., the graph with the vertex set $V(G)\cup V(H)$ and the edge set $E(G)\cup E(H)$.
	
	For each $n\geq 3$, we shall use $P_n$ and $C_n$ to denote the path graph and cycle graph on $n$ vertices, respectively. A \emph{tree} is a connected graph with no induced cycle.  Equivalently, a tree is a graph where there is a unique path between any two of its vertices. For this reason, if $G$ is a tree and $x,y$ are two different vertices of $G$, let $\mathcal{P}(x,y)$ denote the (unique) path from $x$ to $y$.
	A longest path in a tree $G$ is called its \emph{diameter}, and we denote its length by $\diam(G)$. We note that in a tree, there can be many longest paths. 
	
	Let $G$ be a finite simple graph. If $xy$ is an edge of $G$, then $x$ is called a \emph{neighbor} of $y$ in $G$, and vice versa. For a vertex $x$ of $G$, let $N_G(x)$ denote the set of all neighbors of $x$ in $G$. The \emph{degree} of $x$ (in $G$) is the number of neighbors of $x$ in $G$. A \emph{leaf vertex} of $G$ is one that is of degree 1.  A \emph{caterpillar tree} is a tree whose vertices with degree at least 2 form a path in $G$; this (unique) path is called the \emph{central path} of $G$.

	\subsection{Linear quotients and linear resolution}
	
	Let $S$ be the polynomial ring $\Bbbk[x_1,\dots, x_m]$ over a field $\Bbbk$, $\mathfrak{m}$ its irrelevant ideal, and $M$ a finitely generated module over $S$. A \emph{free resolution} of $M$ over $S$ is a complex of free $S$-modules
	\[
	\mathcal{F}\colon \cdots \to F_r\xrightarrow{\partial} F_{r-1} \to \cdots \to F_1 \xrightarrow{\partial} F_0\to 0
	\]
	such that $H_0(\mathcal{F})\cong M$ and $H_i(\mathcal{F})=0$ for any $i>0$. Moreover, $\mathcal{F}$ is called \emph{minimal} if $\partial(F_{i+1})\subseteq \mathfrak{m}F_{i}$ for any $i$. It is known that the minimal free resolution of $M$ is finite, i.e., $F_i=0$ if $i\gg 0$.
	
	In the case where $M$ is $\mathbb{N}$-graded, it is well-known that it has an $\mathbb{N}$-graded minimal free resolution. In other words, the free modules $F_i$ can be given a shift so that the differentials are homogeneous. We can thus set $F_i=\oplus_{j\in \mathbb{N}} S(-j)^{\beta_{i,j}(M)}$ for any integer $i$. The numbers $\beta_{i,j}(M)$ are called \emph{Betti numbers} of $M$. The \emph{(Castelnuovo-Mumford) regularity} of $M$, denoted by $\reg M$, is defined to be
	\[
	\reg M\coloneqq \max \{ j - i \colon \beta_{i,j}(M)\neq 0 \}.
	\]
	In this article we will only consider the case where $M=I$ is a monomial ideal of $S$. Recall that $I$ has a unique minimal monomial generating set, which we denote by $\mingens(I)$. If $I$ is generated by monomials in the same degree $d$, we say that $I$ is \emph{equigenerated} in degree $d$. A monomial ideal $I$ equigenerated in degree $d$ is said to have \emph{linear resolution} if $\reg I = d$. We say that $I$ has \emph{linear quotients} if after a relabelling, we have $\mingens(I)=(m_1, m_2,\dots, m_q)$ where the colon ideal
	\[
	(m_1,m_2,\dots, m_k)\colon (m_{k+1}) 
	\]
	is generated by a set of variables of $S$, for any $k\in [q-1]$. The following result is well-known.
	
	\begin{lemma}[{\cite[Theorem 8.2.15]{HerzogHibiBook}}]\label{lem:LQ-implies-LR}
		Equigenerated monomial ideals with linear quotients have linear resolution.
	\end{lemma}
	
	For a monomial ideal $I$ and a monomial $m$, let $I^{\leq m}$ be the monomial ideal generated by monomials in $\mingens(I)$ that divide $m$. The following is a direct corollary of the well-known Restriction Lemma.
	
	\begin{lemma}[{\cite[Lemma 4.1]{HHZ2004}}]\label{lem:linear-resolution-induced}
		If a monomial ideal $I$ has linear resolution, so does $I^{\leq m}$ for any monomial $m$.
	\end{lemma}
	
	In general, an analog of the Restriction Lemma for the property of having linear quotients does not hold. However, it does when the ideal is equigenerated.
	
	\begin{lemma}[{\cite[Proposition 2.6]{HMRG20}}]\label{lem:linear-quotients-restriction}
		If an equigenerated monomial ideal $I$ has linear quotients, then so does $I^{\leq m}$ for any monomial $m$.
	\end{lemma}
	
	Our focus is on the class of $n$-path ideals $J_n(G)$ where $n\geq 2$ is an integer. It is clear from definition that if $H$ is an induced subgraph of $G$, then $J_n(G)^{\leq \prod_{x\in V(H)} x} = J_n(H)$. Thus the following follows immediately from Lemma~\ref{lem:linear-quotients-restriction}.
	
	\begin{lemma}\label{lem:linear-quotients-induced}
		Let $G$ be a graph and $n\geq 2$ a positive integer. If $J_n(G)$ has linear quotients, then so does $J_n(H)$ for any induced subgraph $H$ of $G$.
	\end{lemma}

	\subsection{Regularity for some monomial ideals}
	
	In view of Lemma~\ref{lem:LQ-implies-LR}, our main tool in showing that a monomial ideal does not have linear quotients, is to show that it does not have linear resolution. For this reason, we compute the regularity of monomial ideals in some special cases in this section.
	
	First we recall the concept of Eliahou-Kervaire splittings: a decomposition $I=J+K$ of a monomial ideal into the sum of two monomial ideals is called an \emph{Eliahou-Kervaire splitting} if $\mingens(I)$ is the disjoint union of $\mingens(J)$ and $\mingens(K)$  and there exists a function
	\begin{align*}
		\phi\colon \mingens(J\cap K) &\to \mingens(J)\times \mingens(K)\\
		m &\mapsto (\phi_1(m),\phi_2(m))
	\end{align*} such that the following holds:
	\begin{enumerate}
		\item for any $m\in \mingens(J\cap K)$, we have $m=\lcm(\phi_1(m),\phi_2(m))$;
		\item for any non-empty subset $\sigma\subseteq \mingens(J\cap K)$, the monomials $\lcm(\phi_1(\sigma))$ and $\lcm(\phi_2(\sigma))$ properly divide $\lcm(\sigma)$.
	\end{enumerate}
	The function $\phi$ in this case is called a \emph{splitting map} of $I$. By \cite[Proposition 3.2]{Fatabbi2001}, if $I=J+K$ is an Eliahou-Kervaire splitting, then 
	\[
	\beta_{i,j}(I)=\beta_{i,j}(J)+\beta_{i,j}(K)+\beta_{i-1,j}(J\cap K) \;\;\;\text{ for any integers } i,j.
	\]
	By definition, we have the following (see also, \cite[Corollary 2.2]{FHV09}).
	
	\begin{lemma}\label{lem:regularity-EK-splitting}
		If $I=J+K$ is an Eliahou-Kervaire splitting, then
		\[
		\reg I = \max \{\reg J, \ \reg K, \ \reg(J\cap K) -1 \}.
		\]
	\end{lemma}
	
	We are now ready to compute the regularity of the following special monomial ideals. We remark that these results are not particularly new, and can be derived from various methods. For example, Lyubeznik resolutions \cite{Ly88} minimally resolve all the monomial ideals below. Instead here we offer short and self-contained proofs for these results using Eliahou-Kervaire splittings.
	
	\begin{lemma}\label{lem:reg-2-gens}
		If $I$ is a monomial ideal where $\mingens(I)=\{m_1,m_2\}$, then 
        \[
        \reg I = \deg (\lcm(m_1,m_2)) -~1.
        \]
	\end{lemma}
	
	\begin{proof}
		Set $J=(m_1)$ and $K=(m_2)$. Then $J\cap K=(\lcm(m_1,m_2))$, and it is clear that $\mingens(J\cap K)=\{\lcm(m_1,m_2)\}$. Consider the map
		\begin{align*}
			\phi\colon \mingens(J\cap K) &\to \mingens(J)\times \mingens(K)\\
			\lcm(m_1,m_2) &\mapsto (m_1,m_2).
		\end{align*}
		We claim that $\phi$ is an Eliahou-Kervaire splitting. The only non-trivial condition we need to check is (2). Consider a non-empty subset $\sigma$ of $\mingens(J\cap K)$. Then we have $\sigma = \mingens(J\cap K)$ itself. Since $\mingens(I)=\{m_1,m_2\}$, these two monomials do not divide each other. In particular, $m_1$ and $m_2$ properly divides $\lcm(\sigma)=\lcm(m_1,m_2)$. Thus the claims holds. By Lemma~\ref{lem:regularity-EK-splitting}, we then have
		\begin{align*}
			\reg I &= \max \{\reg J, \ \reg K, \ \reg(J\cap K) -1 \}\\
			&= \max \{\deg(m_1), \ \deg(m_2), \ \deg(\lcm(m_1,m_2)) -1 \}.
		\end{align*}
		Since $m_1$ and $m_2$ do not divide each other, $\deg(\lcm(m_1,m_2))$ is bigger than $\deg(m_1)$ and $\deg(m_2)$. The result then follows.
	\end{proof}

	\begin{lemma}\label{lem:reg-3-gens}
		If $I$ is a monomial ideal with $\mingens(I)=\{m_1,m_2,m_3\}$ where $m_3\mid \lcm(m_1,m_2)$, then 
		\[
		\reg I = \max\{ \deg \lcm(m_1,m_3), \deg \lcm(m_2,m_3) \} -1.
		\]
	\end{lemma}
	
	\begin{proof}
		Set $J=(m_1,m_3)$ and $K=(m_2)$. Then
		\[
		J\cap K = (\lcm(m_1,m_2),\lcm(m_2,m_3)) = (\lcm(m_2,m_3))
		\]
		as $m_3\mid \lcm(m_1,m_2)$. Consider the map
		\begin{align*}
			\phi\colon \mingens(J\cap K) &\to \mingens(J)\times \mingens(K)\\
			\lcm(m_2,m_3) &\mapsto (m_3,m_2).
		\end{align*}
		The fact that $\phi$ is an Eliahou-Kervaire splitting follows from similar arguments in the proof of Lemma~\ref{lem:reg-2-gens}. By Lemmas~\ref{lem:regularity-EK-splitting} and \ref{lem:reg-2-gens}, we then have
		\begin{align*}
			\reg I &= \max \{\reg J, \ \reg K, \ \reg(J\cap K) -1 \}\\
			&= \max \{\deg(\lcm(m_1,m_3)) -1, \ \deg(m_2), \ \deg(\lcm(m_2,m_3)) -1 \}\\
			&= \max \{\deg(\lcm(m_1,m_3)) -1, \ \deg(\lcm(m_2,m_3)) -1 \},
		\end{align*}
		where the last equality is due to $m_2$ and $m_3$ do not divide each other. The result then follows.
	\end{proof}

\section{Forbidden structures}\label{sec:forbidden}
	
	In this section, we identify certain classes of graphs that are obstructions to the linearity of the minimal free resolution of path ideals. It is clear (from different methods) that if a graph $G$ contains $P_n+P_n$ as an induced subgraph, then $J_n(G)$ does not have linear resolution.
	\begin{lemma}\label{lem:forbidden-Pn+Pn}
		Let $n\geq 2$ be a positive integer. Then $J_n(P_n+P_n)$ does not have linear resolution.
	\end{lemma}
	
	\begin{proof}
		We can assume that $J_n(G)= \left( \prod_{i=1}^n x_i, \prod_{j=1}^n y_j \right)$. Then by Lemma~\ref{lem:reg-2-gens}, we have $\reg J_n(G) = 2n -1 >n$, or in particular, $J_n(G)$ does not have linear resolution, as desired.
	\end{proof}
	
	For $n\geq 5$ and $k\in [3,n-2]$, define $L_{n,k}$ to be the graph with the vertex set
	\[
	V(L_{n,k}) = \{x_1,x_2,\dots, x_n\} \cup \{y_1,y_2,\dots , y_{k-1}\}
	\]
	and the edge set
	\[
	E(L_{n,k}) = \{x_ix_{i+1}\colon i\in [n-1]\} \cup\{y_iy_{i+1}\colon i\in [k-2]\} \cup \{y_{k-1}x_k\}.
	\]
	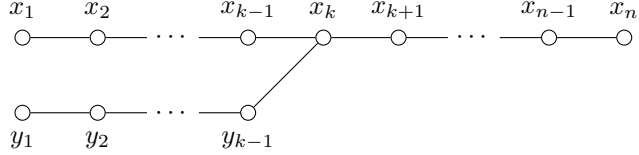
\begin{figure}[!hbt]
		
		\begin{center}
			\begin{tikzpicture}[every node/.style={circle,draw,inner sep=2pt}, every label/.append style={shape= rectangle, minimum height=5mm, align=center}]
				\foreach \i in {1,2} {
					\node[label={90:$x_{\i}$}] (x\i) at (\i, 0) {};
				}
				\node[draw=none] (dots1) at (3,0) {$\cdots$};
				\node[label={90:$x_{k-1}$}] (xk-1) at (4,0) {};
				\node[label={90:$x_k$}] (xk) at (5,0) {};
				\node[label={90:$x_{k+1}$}] (xk+1) at (6,0) {};
				\node[draw=none] (dots2) at (7,0) {$\cdots$};
				\node[label={90:$x_{n-1}$}] (xn-1) at (8,0) {};
				\node[label={90:$x_n$}] (xn) at (9,0) {};
				\foreach \i in {1,2} {
					\node[label={270:$y_{\i}$}] (y\i) at (\i,-1) {};
				}
				\node[draw=none] (dots3) at (3,-1) {$\cdots$};
				\node[label={270:$y_{k-1}$}] (yk-1) at (4,-1) {};
				\draw (x1) -- (x2) -- (dots1) -- (xk-1) -- (xk) -- (xk+1) -- (dots2) -- (xn-1) -- (xn);
				\draw (y1) -- (y2) -- (dots3) -- (yk-1) -- (xk);
			\end{tikzpicture}
		\end{center}
		\label{fig:Diagram of L_{n,k}}
		\caption{The graph $L_{n,k}$}
	\end{figure}
	
	 The graph $L_{n,k}$ is a tree. If $k\in[ (n+1)/2, n-2]$, it is straightforward that the induced subgraph of $L_{n,k}$ on the vertex set $x_1,\dots , x_n, y_{2k-n}, y_{n-k+1},\dots , y_{k-1}$ is isomorphic to $L_{n,n-k+1}$ where $n-k+1\in [3,(n+1)/2]$. This implies the following result.
    \begin{lemma}\label{lem:minimal-Lnk}
        A graph $G$ does not contain $L_{n,k}$, where $k\in [3,(n+1)/2]$, as an induced subgraph if and only if it does not contain $L_{n,k}$, where $k\in [3,n-2]$, as an induced subgraph.
    \end{lemma}
    This (trivial) result means that the only (minimal) forbidden structures as an induced subgraph among $L_{n,k}$, where $k\in [3,n-2]$, are $L_{n,k}$, where $k\in [3,(n+1)/2]$. However, the former is easier to use (as shall be seen in the sequel) due to its symmetry.

    Now assume that $k\in [3,(n+1)/2]$. There are three paths in $L_{n,k}$ that are candidates of being the longest paths, namely those that connect two among the vertices $x_1, y_1,$ and $x_n$. The two paths that connect $x_n$ to $x_1$ and $y_1$ are both of length $n-1$. On the other hand, the path connecting $x_1$ to $y_1$ is of length $2(k-1)\leq n-1$. Therefore, the ideal $J_n(L_{n,k})$ has at least two minimal generators
    \[
    \prod_{i=1}^n x_n\quad \text{ and }\quad \left(\prod_{i=1}^{k-1} y_i\right)\left( \prod_{j=k}^n x_j\right)
    \]
    and potentially one more: $\left(\prod_{i=1}^{k-1} y_i\right)\left( \prod_{j=1}^k x_j\right)$ exactly when $k=(n+1)/2$ (in this case, $n$ must be odd). Now, we show that the class of graphs $L_{n,k}$ where $ k\in [3,(n+1)/2]$ are forbidden induced subgraphs of $G$ when $J_n(G)$ have linear resolution.

	\begin{lemma}\label{lem:forbidden-structures}
		Let $G=L_{n,k}$ where $n\geq 5$  and $k\in [3,(n+1)/2]$. Then $J_n(G)$ does not have linear resolution.
	\end{lemma}
	\begin{proof}
		If $k<(n+1)/2$, then 
        \[
        J_n(G)=\left(\prod_{i=1}^n x_n,\  \left(\prod_{i=1}^{k-1} y_i\right)\left( \prod_{j=k}^n x_j\right)\right),
        \]
        and thus $\reg J_n(G) = n+k-2 \geq n+1$ by Lemma~\ref{lem:reg-2-gens}. On the other hand, if $k=\frac{n+1}{2}$, then
        \[
        J_n(G)=\left(\prod_{i=1}^n x_n,\  \left(\prod_{i=1}^{k-1} y_i\right)\left( \prod_{j=k}^n x_j\right), \ \left(\prod_{i=1}^{k-1} y_i\right)\left( \prod_{j=1}^k x_j\right)\right),
        \]
        and thus $\reg J_n(G) = \max\{ n+(k-1)-1, n+(k-1)-1 \} = n+k-2 \geq n+1$ by Lemma~\ref{lem:reg-3-gens}. In both cases, it follows that $J_n(G)$ does not have linear resolution.
	\end{proof} 
    Interestingly, while $J_4(G)$ having linear resolution has a forbidden structure other than $P_4+P_4$, it does not follow the same rule as $J_n(G)$ for $n\geq 5$. In fact, $L_{5,3}$ serves as a (minimal) forbidden structure for both $J_4(G)$ and $J_5(G)$ having linear resolution.

	\begin{lemma}\label{lem:forbidden-for-J4}
		The ideal $J_4(L_{5,3})$ does not have linear resolution. 
	\end{lemma}
	
	\begin{proof}
		It is straightforward that 
		\[
		J_4(L_{5,3})= x_3(x_1x_2x_4,x_1x_2y_2,x_2x_4x_5,x_4x_5y_2,x_2y_1y_2,x_4y_1y_2).\]
		Set
		\begin{align*}
			J\coloneqq x_1x_2x_3(x_4,y_2) \quad \text{and} \quad 
			K\coloneqq x_3(x_2x_4x_5,x_4x_5y_2,x_2y_1y_2,x_4y_1y_2).
		\end{align*}
		It is clear that $J$ is generated by all minimal generators of $J_4(L_{5,3})$ that are divisible by $x_1$, and $K$ is generated by the remaining generators of $J_4(L_{5,3})$. Since $J$ has linear resolution, by \cite[Corollary 2.7]{FHV09} and Lemma~\ref{lem:reg-3-gens}, we have 
		\[
		\reg J_4(L_{5,3}) \geq \reg(J\cap K) - 1 = \reg \left( x_1x_2x_3(x_4x_5, y_1y_2) \right) -1 = 5.
		\]
		In particular, $J_4(L_{5,3})$ does not have linear resolution, as desired.
	\end{proof}
	
    For $n=2,3$, the only forbidden structure the lemmas in this section provide is $P_n+P_n$ . As it turns out, this is the only forbidden structure in these cases, given that the graph is a tree, as discussed in the introduction (see also, \cite[Proposition 3.9]{Zheng2004}, \cite[Theorem 3.2]{HHZ2004}, and  \cite[Theorem 5.1]{AJM2024}). Therefore, for the rest of article, we assume that $n\geq 4$. We will often consider the case when $G$ does not contain the forbidden structures we have found. Moreover, we make some observations.
    \begin{observation}\label{obser}
     Let $G$ be a graph and $n\geq 2$ an integer. Then
        \begin{itemize}
            \item $J_n(G)\neq (0)$ if and only if $\diam(G)\geq n-1$;
            \item for a tree $G$ which does not contain $P_n+P_n$ as an induced subgraph, one has $\diam(G)\leq 2n-1$.
        \end{itemize}
    \end{observation}

    The first observation follows from definition. The second observation can be shown using contraposition: let $G$ be a tree with $\diam(G)\geq  2n$; then $G$ contains a path $x_1,\dots, x_{2n+1}$, and the induced subgraph of $G$ with the vertex set $\{x_1,\dots, x_{n},x_{n+2},\dots, x_{2n+1}\}$ is $P_n+P_n$, as~desired.
    
    In light of our observations, we set up the following notation. Here we use the letter F for these conditions, where F is short for ``forbidden".

    \begin{notation}\normalfont\label{F(n) condition} 
	    We assume that $G$ is a tree. Set
	\begin{align*}
		\diam(G)\in [3,7] \text{ and } G \text{ does not contain $P_4+P_4$ or $L_{5,3}$ as an induced subgraph} \tag{$F_4$}
	\end{align*}
	and for any $n\geq 5$,
	\begin{align*}
		&\begin{multlined}[t]
			\diam(G)\in [n-1,2n-1] \text{ and } G \text{ does not contain $P_n+P_n$ or $L_{n,k}$}\\
			\text{ as an induced subgraph for any $k\in [3,\ n-2].$}
		\end{multlined} \tag{$F_n$}
	\end{align*}
	\end{notation}

	\section{Trimmed trees}\label{sec:trimmed trees}
	

	The main objective of this section is to understand the structure of a tree if it does not contain the forbidden structures described in the previous section. Given an integer $n\geq 4$. Assume that $G$ either is a caterpillar tree or satisfies the $(F_n)$ condition. We define the following operation on $G$, which we call \emph{trimming}. For a longest path $v_1, v_2,\dots, v_{\diam(G)+1}$ in $G$, we define the \emph{trimmed graph} of $G$ (w.r.t the path $v_1, v_2,\dots, v_{\diam(G)+1}$), denoted by $\trim(G,v_1,v_2,\dots, v_{\diam(G)+1})$, to be the induced subgraph of $G$ with the vertex set $\bigcup_{i=1}^{\diam(G)+1} N_G[v_i]$. It is important to note that this operation depends on a graph $G$ and a longest path of $G$. Therefore,  it is somewhat surprising that this operation only depends on $G$ when $G$ has the $(F_n)$ condition for some $n\geq 4$, as shall be shown in Corollary~\ref{cor:trim-well-defined}. It is worth noting that $\trim(G,v_1,v_2,\dots, v_{\diam(G)+1})$ is a caterpillar graph by construction. Before continuing, we fix some notations.
	\begin{notation}\label{nota:caterpillar}\normalfont
		Let $G$ be a caterpillar tree of diameter $d$, and $x_1,\dots, x_{d-1}$ the induced path formed by all vertices of degree $2$ or more of $G$. By definition, all other vertices of $G$ are of degree $1$, and since $G$ is connected, any vertex in $V(G)\setminus \{x_1,\dots, x_{d-1}\}$ must be incident to exactly one vertex in $\{x_1,\dots, x_{d-1}\}$. In other words, we have $V(G)=\{x_1,\dots, x_{d-1}\}\cup \left(\bigcup_{i=1}^{d-1} N_G(x_i) \right)$. It would be easier to work with a disjoint union. Therefore, for any $i\in[1,d-1]$, we set $LN_G(x_i)\coloneqq N_G(x_i)\setminus \{x_{i-1},x_{i+1}\}$, called the set of \emph{leaf neighbors} of $x_i$. Then
		\begin{align*}
		    V(G)&= \{x_1,\dots, x_{d-1}\}\sqcup \left(\bigsqcup_{i=1}^{d-1} LN_G(x_i) \right),\\
		E(G)&= \{ x_ix_{i+1}\colon i\in [1,d-2] \}\sqcup \{x_iy\colon i\in [1,d-1], y\in  LN_G(x_i)\}.
		\end{align*}

		For each $i\in [1,d-1]$, we set $LN_G(x_i)\eqqcolon\{x_{i1},\dots, x_{il_i}\}$ where $l_1\geq 1, l_{d-1}\geq 1$, and $l_i\geq 0$ if $i\in[2,d-2]$.
	\end{notation}
	
	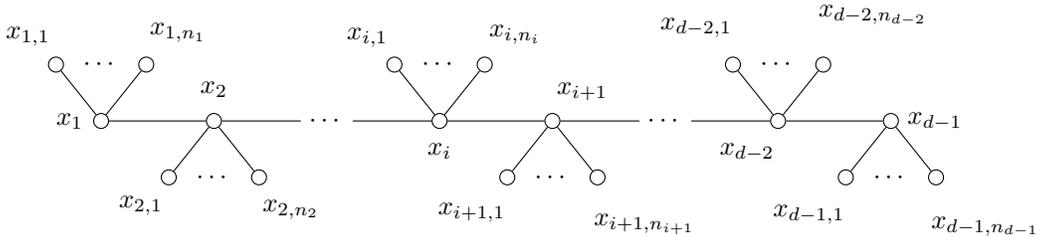
\begin{figure}[H]
		\begin{center}
			\begin{tikzpicture}[every node/.style={circle,draw,inner sep=2pt}, scale=1.5]
				
				
				\node[label={180: $x_1$}] (x1) at (0,0) {} ;
				\node[label={90: $x_2$}] (x2) at (1,0) {};
				\node[label={270:$x_{i}$}] (x3) at (3,0) {};
				\node[label={45: $x_{i+1}$}] (x4) at (4,0) {};
				\node[label={225: $x_{d-2}$}] (x5) at (6,0) {};
				\node[label={0: $x_{d-1}$}] (x6) at (7,0) {};
				\node[draw=none] (dot1) at (2,0) {$\cdots$};
				\node[draw=none] (dot2) at (5,0) {$\cdots$};
				\node[draw=none] (dot3) at (0,0.5) {$\cdots$};
				\node[draw=none] (dot4) at (3,0.5) {$\cdots$};
				\node[draw=none] (dot5) at (6,0.5) {$\cdots$};
				\node[draw=none] (dot6) at (1,-0.5) {$\cdots$};
				\node[draw=none] (dot7) at (4,-0.5) {$\cdots$};
				\node[draw=none] (dot8) at (7,-0.5) {$\cdots$};
				
				\node[label={135: $x_{1,1}$}] (u1) at (-0.4,0.5) {} ;
				\node[label={45: $x_{1,n_1}$}] (u2) at (0.4,0.5) {} ;
				\node[label={135: $x_{i,1}$}] (u3) at (2.6,0.5) {} ;
				\node[label={45: $x_{i,n_i}$}] (u4) at (3.4,0.5) {} ;
				\node[label={135: $x_{d-2, 1}$}] (u5) at (5.6,0.5) {} ;
				\node[label={45: $x_{d-2, n_{d-2}}$}] (u6) at (6.4,0.5) {} ;
				
				\node[label={225: $x_{2,1}$}] (l1) at (0.6,-0.5) {} ;
				\node[label={315: $x_{2,n_2}$}] (l2) at (1.4,-0.5) {} ;
				\node[label={225: $x_{i+1,1}$}] (l3) at (3.6,-0.5) {} ;
				\node[label={315: $x_{i+1,n_{i+1}}$}] (l4) at (4.4,-0.5) {} ;
				\node[label={225: $x_{d-1, 1}$}] (l5) at (6.6,-0.5) {} ;
				\node[label={315: $x_{d-1, n_{d-1}}$}] (l6) at (7.4,-0.5) {} ;
				
				\draw (x1) -- (x2);
				\draw (x3) -- (x4);
				\draw (x5) -- (x6);
				\draw (x2) -- (dot1);
				\draw (dot1) -- (x3);
				\draw (x4) -- (dot2);
				\draw (dot2) -- (x5);
				\draw (u1) -- (x1);
				\draw (u2) -- (x1);
				\draw (u3) -- (x3);
				\draw (u4) -- (x3);
				\draw (u5) -- (x5);
				\draw (u6) -- (x5);
				\draw (l1) -- (x2);
				\draw (l2) -- (x2);
				\draw (l3) -- (x4);
				\draw (l4) -- (x4);
				\draw (l5) -- (x6);
				\draw (l6) -- (x6);
			\end{tikzpicture}\\
			
		\end{center}
		\label{fig:caterpillar_G}
		\caption{A caterpillar graph $G$}
	\end{figure}

	\begin{lemma}\label{lem:2-longest-paths}
		Let $G$ be either a caterpillar tree or a tree that satisfies the $(F_n)$ condition for some integer $n\geq 4$. Set $\diam(G)=d$. Then for any two longest paths $z_1,z_2,\dots, z_{d+1}$ and  $w_1,w_2,\dots, w_{d+1}$ of $G$, either of the following holds:
		\begin{enumerate}
			\item $z_2=w_2$ and $z_{d}=w_{d}$.
			\item $z_2=w_{d}$ and $z_{d}=w_2$.
		\end{enumerate}
	\end{lemma}
	\begin{proof}
		Assume that $G$ is a caterpillar graph. Let $\mathcal{P}\colon x_1,\dots, x_{d-1}$ be the central path of $G$ formed by all of its vertices of degree 2 or more. Note that any path $P_n$ in $G$ consists of at least $n-2$ vertices of degree 2 or more. Therefore, any any longest path of $G$ contains $x_1,\dots, x_{d-1}$, and its two ends are a neighbor of $x_1$ and a neighbor of $x_{d-1}$. Thus, the assertion follows.
		

		
		Assume $n\in \{4,5\}$ and $G$ satisfies the $(F_n)$ condition. In particular, $G$ does not contain $L_{5,3}$ as an induced subgraph. Thus $G$ is a caterpillar tree by \cite[Theorem~1]{HS71}, and the result then follows from the previous case. Now we can assume that $n\geq 6$ and $G$ satisfies the $(F_n)$ condition. We have the following claim.
		\begin{claim}\label{clm:z-w-intersect}
			There exist integers $r,s\in [d+1]$ such that $z_r=w_s$.
		\end{claim}
		\begin{proof}[Proof of Claim~\ref{clm:z-w-intersect}]
			Suppose not, i.e., the two paths $z_1,z_2,\dots, z_{d+1}$ and $w_1,w_2,\dots, w_{d+1}$ have no vertex in common. We will derive a contradiction. The path $\mathcal{P}(z_1,w_1)$ must be of the form
			\[
			\mathcal{P}(z_1,w_1)\colon z_1,z_2, \dots, z_u, y_1,\dots, y_w, w_v,w_{v-1},\dots ,w_1 
			\]
			By flipping the indices on the path $z_1,z_2,\dots, z_{d+1}$ if needed, we can assume that $u\geq \frac{d+2}{2}$. Similarly, we can assume that $v\geq \frac{d+2}{2}$. Then we have
			\[
			\length(\mathcal{P}(z_1,w_1)) = u+w+v-1 \geq (d+2) +w -1 = d+w+1 > d,
			\]
			a contradiction, as desired.
		\end{proof}
		If $r=1$, then since $w_s=z_1$ is a leaf, we have $s=1$ or $s=d+1$. By flipping the indices on the path $w_1,w_2,\dots, w_{d+1}$ if needed, we can assume that $s=1$. We then have $z_2=w_2$ as well since $z_1$ is incident to only $z_2$, and $w_1$ to $w_2$. Let $a$ be the largest integer where $z_i=w_i$ for any $i\leq a$. If $a\geq d$, then $w_d=z_d$, as desired. Now we can assume that $2\leq a\leq d-1$. We will derive a contradiction. We first observe that $\{z_{j}\colon j>a\} \cap \{w_{j}\colon j>a\}=\emptyset$, as otherwise it would contradict the definition of $a$ (in the case $z_{a+1}=w_{a+1}$ is in the intersection), or create a cycle in a tree $G$, also a contradiction. If $a=2$, then the path $\mathcal{P}(z_{d+1},w_{d+1})$ is of the form
		\[
		\mathcal{P}(z_{d+1},w_{d+1})\colon z_{d+1},z_d,\dots, z_2,w_3,w_4,\dots, w_{d+1} 
		\]
		and hence is of length $$2(d-1)=d+(d-2)\geq d+(n-1-2)>d,$$
        where the second to last inequality is due to the $(F_n)$ condition. This is a contradiction. Therefore, we can now assume that $3\leq a\leq d-1$. Consider the induced subgraph of $G$ with the vertex set $\{z_i,w_i\colon i\in [d+1]\}$:
		
		\begin{figure}[H]
			\begin{center}
				\begin{tikzpicture}[every node/.style={circle,draw,inner sep=2pt}, scale=1.5]

					\node[label={90: $z_1$}] (z1) at (0,0) {} ;
					\node[label={90: $z_2$}] (z2) at (1,0) {};
					\node[label={90:$z_{a-1}$}] (za-1) at (3,0) {};
                    \node[label={90:$w_a=z_a$}] (za) at (4,0) {};
					\node[label={270:$w_{a-1}$}] (wa-1) at (3,-0.5) {};
					\node[label={270:$w_{a-2}$}] (wa-2) at (2,-0.5) {};
					\node[draw=none] (wdot) at (1,-0.5) {$\cdots$};
					\node[draw=none] (dot1) at (2,0) {$\cdots$};
					\node[draw=none] (dot2) at (5,0) {$\cdots$};
					\node[label=90:$z_d$] (zd) at (6,0) {};
					\node[label=90:$z_{d+1}$] (zd+1) at (7,0) {};
					\node[label=270:$w_1$] (w1) at (0,-0.5) {};
					
					\draw (z1) -- (z2);
					\draw (z2) -- (dot1);
					\draw (dot1) -- (za-1) -- (za) -- (dot2)-- (zd);
					\draw (za) -- (wa-1);
					\draw (wa-1) -- (wa-2);
					\draw (wa-2) -- (wdot);
					\draw (wdot) -- (w1);
					\draw (zd) -- (zd+1);
				\end{tikzpicture}
			\end{center}
			\caption{$L_{d+1,a}$ is induced by $\{z_i,w_i: i\in[d+1]\}$}
		\end{figure}
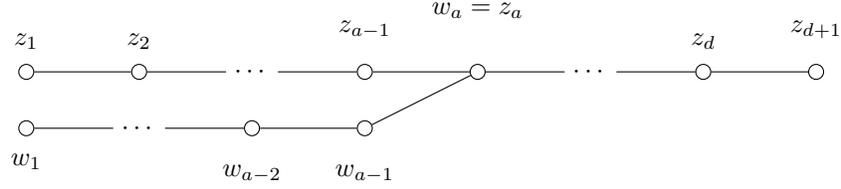
		
		It is clear that this graph is isomorphic to $L_{d+1,a}$, where $a\in [3,d-1]$. Since $d+1\geq n$ (due to the $(F_n)$ condition), it contains $L_{n,a}$ as an induced subgraph. This contradicts the fact that $G$ satisfies the $(F_n)$ condition, as desired.
	\end{proof}
	
	\begin{corollary}\label{cor:trim-well-defined}
		Let $G$ be a caterpillar tree or a tree that satisfies the $(F_n)$ condition where $n\geq 4$ is a positive integer. Then for any two longest paths $z_1,z_2,\dots, z_{d+1}$ and $w_1,w_2,\dots, w_{d+1}$ of $G$, we have $\trim(G,z_1,z_2,\dots, z_{d+1}) = \trim(G,w_1,w_2,\dots, w_{d+1})$. In other words, in this case, $\trim$ depends only on $G$. 
	\end{corollary}
	\begin{proof}
		By flipping the indices if necessary, Lemma~\ref{lem:2-longest-paths} implies that $z_2=w_2$ and $z_d=w_d$. Therefore we also have $z_i=w_i$ for any $i\in [2,d]$ as otherwise $G$ would contain a cycle. By definition, the two graphs $\trim(G,z_1,z_2,\dots, z_{d+1})$ and $\trim(G,w_1,w_2,\dots, w_{d+1})$ are then exactly the same, as desired.
	\end{proof}
	
	Due to this result, we will simply use $\trim(G)$ when $G$ is a caterpillar tree or a tree that satisfies the $(F_n)$ condition where $n\geq 4$. It is worth noting that if $G$ is a caterpillar graph, then it is straightforward that $\trim(G)=G$.
	
	The following lemma is the key to our results, essentially reducing our problem from trees to caterpillar trees.
	
	
	\begin{lemma}\label{lem:G->trim(G)}
		Let $G$ be a tree that satisfies the $(F_n)$ condition where $n\geq 4$ is a positive integer. Then $J_n(G)=J_n(\trim(G))$.
	\end{lemma}
	\begin{proof}
		If $n\in \{4,5\}$ then the $(F_n)$ condition implies that $G$ is a caterpillar tree by \cite[Theorem~1]{HS71}. Then $\trim(G)=G,$ and the result follows trivially. Now we can assume that $n\geq 6$.
		
		Set $d=\diam(G)$. Let $w_1,w_2,\dots, w_{d+1}$ be a longest path in $G$.  
		It is straightforward to see that $\mingens(J_n(\trim(G)))\subseteq \mingens(J_n(G))$. Suppose that we do not have the equality. Then there exists a path $z_1,z_2,\dots, z_n$ in $G$ where $z_p\notin N_G(w_i)$  for any $i\in [d+1]$, for some $p\in [n]$. We have the following claim.
		\begin{claim}\label{clm:n-path-does-not-intersect-longest-path}
			$\{w_1,w_2,\dots , w_{d+1}\}\cap \{z_1,z_2,\dots, z_n\}\neq \emptyset$.
		\end{claim}
		\begin{proof}[Proof of Claim~\ref{clm:n-path-does-not-intersect-longest-path}]
			Suppose otherwise that $\{w_1,w_2,\dots , w_{d+1}\}\cap \{z_1,z_2,\dots, z_n\}= \emptyset$. If there is no edge that connects a vertex in $\{w_1,w_2,\dots, w_{d+1}\}$ to a vertex in $\{z_1,z_2,\dots, z_n\}$, then the induced subgraph of $G$ with the vertex set $\{w_1,w_2,\dots, w_n\}\cup \{z_1,z_2,\dots, z_n\}$ is $P_n+P_n$, a contradiction. Thus we can assume that there is an edge $w_sz_r\in E(G)$ for some $s\in [d+1]$ and $r\in [n]$. By flipping the indices on the two paths if needed, we can assume that $s\geq \frac{d+2}{2}$ and $r\leq \frac{n+1}{2}$. First observe that
			\begin{equation*}
				s\geq \frac{d+2}{2}\geq \frac{(n-1)+2}{2}=\frac{n+1}{2}\geq r. \tag{1}
			\end{equation*}
			
            Now, if $r\geq 3$, then we can consider the induced subgraph of $G$ with the vertex set $\{z_1,z_2,\dots, z_n\}\cup \{w_{s-r+2}, w_{s-r+3}, \dots, w_s\}$:
			
			\begin{figure}[H]
			\begin{center}
				\begin{tikzpicture}[every node/.style={circle,draw,inner sep=2pt}, scale=1.5]

					\node[label={90: $z_1$}] (z1) at (0,0) {} ;
					\node[label={90: $z_2$}] (z2) at (1,0) {};
					\node[label={90:$z_{r-1}$}] (za-1) at (3,0) {};
                    \node[label={90:$z_r$}] (za) at (4,0) {};
					\node[label={270:$w_{s}$}] (wa-1) at (3,-0.5) {};
					\node[label={270:$w_{s-1}$}] (wa-2) at (2,-0.5) {};
					\node[draw=none] (wdot) at (1,-0.5) {$\cdots$};
					\node[draw=none] (dot1) at (2,0) {$\cdots$};
					\node[draw=none] (dot2) at (5,0) {$\cdots$};
					\node[label=90:$z_{n-1}$] (zd) at (6,0) {};
					\node[label=90:$z_{n}$] (zd+1) at (7,0) {};
					\node[label=270:$w_{s-r+2}$] (w1) at (0,-0.5) {};
					
					\draw (z1) -- (z2);
					\draw (z2) -- (dot1);
					\draw (dot1) -- (za-1) -- (za) -- (dot2)-- (zd);
					\draw (za) -- (wa-1);
					\draw (wa-1) -- (wa-2);
					\draw (wa-2) -- (wdot);
					\draw (wdot) -- (w1);
					\draw (zd) -- (zd+1);
				\end{tikzpicture}
			\end{center}
			\caption{The subgraph induced by the vertex set $\{z_i,w_j: i\in[n], j\in [d+1]\}$}
		\end{figure}

            It is clear that this graph is isomorphic to $L_{n,r}$ where we already know that $r\in [3,(n+1)/2]$. This contradicts the assumption that $G$ satisfies the $(F_n)$ condition.

            Next, we assume that $r\leq 2$. Then, we must have $s\leq d-1$, as otherwise, consider the path  $w_1,
            w_2, \dots , w_{s}, z_r, z_{r+1}, \dots ,z_n$, which is of length $s+n-r\geq d+n-r>d$, as $r<n$. This is a contradiction to the fact that $\diam(G)=d$.  As $n-1\leq d$, the longest path $w_1,
            w_2, \dots , w_{d+1}$ contains a path of length $n-1$, say $y_1, y_{2}, \dots , y_{n}$. Moreover, since $\frac{d+2}{2}\leq s\leq d-1$, we can choose such a path where $y_t=w_s$ and $t\in [3,n-2]$. Then we can consider the induced subgraph with the vertex set $\{y_1,y_2,\dots , y_n, z_r, z_{r+1}, \dots , z_{n}\}:$

            \begin{figure}[H]
			\begin{center}
				\begin{tikzpicture}[every node/.style={circle,draw,inner sep=2pt}, scale=1.5]

					\node[label={90: $y_1$}] (z1) at (0,0) {} ;
					\node[label={90: $y_2$}] (z2) at (1,0) {};
					\node[label={90:$y_{t}$}] (za-1) at (3,0) {};
                    \node[label={90:$y_{t+1}$}] (za) at (4,0) {};
					\node[label={270:$z_{r}$}] (wa-1) at (4,-0.5) {};
					\node[label={270:$z_{r+1}$}] (wa-2) at (5,-0.5) {};
					\node[draw=none] (wdot) at (6,-0.5) {$\cdots$};
					\node[draw=none] (dot1) at (2,0) {$\cdots$};
					\node[draw=none] (dot2) at (5,0) {$\cdots$};
					\node[label=90:$y_{n-1}$] (zd) at (6,0) {};
					\node[label=90:$y_{n}$] (zd+1) at (7,0) {};
					\node[label=270:$z_{n}$] (w1) at (7,-0.5) {};
					
					\draw (z1) -- (z2);
					\draw (z2) -- (dot1);
					\draw (dot1) -- (za-1) -- (za) -- (dot2)-- (zd);
					\draw (za-1) -- (wa-1);
					\draw (wa-1) -- (wa-2);
					\draw (wa-2) -- (wdot);
					\draw (wdot) -- (w1);
					\draw (zd) -- (zd+1);
				\end{tikzpicture}
			\end{center}
			\caption{The subgraph induced by the vertex set $\{z_i,w_j: i\in[n], j\in [d+1]\}$}
		\end{figure}
            
            It is clear that this graph contains a induced subgraph isomorphic to $L_{n,t}$ for some $t\in [3,n-2]$. This contradicts the assumption that $G$ satisfies the $(F_n)$ condition.
		\end{proof}
		
		By the above claim, the induced subgraph of $G$ with the vertex set $\{w_i,z_j\colon i\in [d+1],j\in[n]\}$ is a tree. Without loss of generality, we can assume that $V(G)=\{w_i,z_j\colon i\in [d+1],j\in[n]\}$, as this tree has the same diameter as the original, and thus the condition $(F_n)$ is preserved. Set
		\begin{align*}
			u&\coloneqq \min \{ i\in [n] \colon z_i\in \{w_1,w_2,\dots, w_{d+1}\} \},\\
			v&\coloneqq \max \{ i\in [n] \colon z_i\in \{w_1,w_2,\dots, w_{d+1}\} \}.
		\end{align*}
		Both $u$ and $v$ exist due to Claim~\ref{clm:n-path-does-not-intersect-longest-path}. We then have $u\leq v$ by their construction, and $z_i\in \{w_1,w_2,\dots, w_{d+1}\}$ for any $i\in [u,v]$, as otherwise $G$ would contain a cycle graph. Also by the definition of $u$ and $v$, we have $z_i\notin \{w_1,w_2,\dots,w_{d+1} \}$ for any $i\in [1,u)\cup (v,n]$. By the definition of $p$, we have $p\in [1,u-2]\cup [v+2,n]$. By flipping the indices if needed, we can assume that $p\in [1,u-2]$. In particular, this implies that $u\geq 3$. 
		
		Let $s$ be the index where $w_s=z_u$. Flipping the indices if needed, we can assume that $s\leq \frac{d+2}{2}$. Consider the induced path $z_1,z_2,\dots, z_u=w_s,w_{s+1},\dots, w_{d+1}$. The length of this path is at most $d$. Hence $u+d-s\leq d$, or equivalently $u\leq s$. We have two cases.    
		
		\textbf{Case 1:} Assume that $u\geq n-1$. We first establish some inequalities. We have
		\[
		n-1\leq u\leq s\leq \frac{d+2}{2},
		\]
		which implies that $d\geq 2n-4$. Moreover, we have
		\[
		s\leq \frac{d+2}{2}\leq \frac{(2n-1)+2}{2} =n+\frac{1}{2},
		\]
		which implies that $s\leq n$. Consider the induced subgraph of $G$ with the vertex set 
        \[
        \{ z_{u-n+4}, z_{u-n+5}, \dots, z_{u-1} \} \cup \{w_{s-n+4}, w_{s-n+5}, \dots, w_{s+3}\}.
        \]

        \begin{figure}[H]
			\begin{center}
				\begin{tikzpicture}[every node/.style={circle,draw,inner sep=2pt}, scale=1.5]

					\node[label={90: $w_{s-n+4}$}] (z1) at (1,0) {} ;
					\node[label={90: $w_{s-n+5}$}] (z2) at (2,0) {};
					\node[label={90:$w_{s-1}$}] (za-1) at (4,0) {};
                    \node[label={90:$w_s=z_s$}] (za) at (5,0) {};
					\node[label={270:$z_{s-1}$}] (wa-1) at (4,-0.5) {};
					\node[label={270:$z_{s-2}$}] (wa-2) at (3,-0.5) {};
					\node[draw=none] (wdot) at (2,-0.5) {$\cdots$};
					\node[draw=none] (dot1) at (3,0) {$\cdots$};
					\node[draw=none] (dot2) at (6,0) {$\cdots$};
					\node[label=90:$w_{s+3}$] (zd+1) at (7,0) {};
					\node[label=270:$z_{u-n+4}$] (w1) at (1,-0.5) {};
					
					\draw (z1) -- (z2);
					\draw (z2) -- (dot1);
					\draw (dot1) -- (za-1) -- (za) -- (dot2)-- (zd+1);
					\draw (za) -- (wa-1);
					\draw (wa-1) -- (wa-2);
					\draw (wa-2) -- (wdot);
					\draw (wdot) -- (w1);
				\end{tikzpicture}
			\end{center}
			\caption{The subgraph induced by the vertex set $\{z_i,w_j: i\in[n], j\in [d+1]\}$}
		\end{figure}
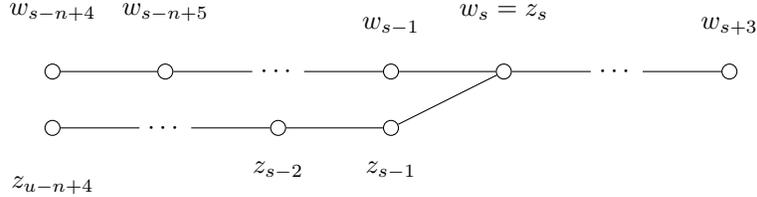
        
		The indices make sense, since
		\begin{align*}
			s+3&\leq n+3 = n+6-3 \leq n+(n)-3=2n-3\leq d+1,\\
			s-n+4 &= (s-n+1)+3 \geq 0+3 \geq 1,\\
			u-n+4 &= (u-n+1) +3 \geq 0 +3\geq 1.
		\end{align*}
		This graph is isomorphic to $L_{n,n-3}$ where $n-3\in [3,n-2]$. This contradicts the assumption that $G$ satisfies the $(F_n)$ condition, as desired.
		
		\textbf{Case 2:} Assume that $u\leq n-2$. Consider the induced subgraph of $G$ with the vertex set $\{ z_1,z_2,\dots, z_n \} \cup \{w_{s-u+1}, w_{s-u+2}, \dots, w_{s-1}\}$: 

        \begin{figure}[H]
			\begin{center}
				\begin{tikzpicture}[every node/.style={circle,draw,inner sep=2pt}, scale=1.5]

					\node[label={90: $z_1$}] (z1) at (0,0) {} ;
					\node[label={90: $z_2$}] (z2) at (1,0) {};
					\node[label={90:$z_{r-1}$}] (za-1) at (3,0) {};
                    \node[label={90:$z_r$}] (za) at (4,0) {};
					\node[label={270:$w_{s-1}$}] (wa-1) at (3,-0.5) {};
					\node[label={270:$w_{s-2}$}] (wa-2) at (2,-0.5) {};
					\node[draw=none] (wdot) at (1,-0.5) {$\cdots$};
					\node[draw=none] (dot1) at (2,0) {$\cdots$};
					\node[draw=none] (dot2) at (5,0) {$\cdots$};
					\node[label=90:$z_{n-1}$] (zd) at (6,0) {};
					\node[label=90:$z_{n}$] (zd+1) at (7,0) {};
					\node[label=270:$w_{s-u+1}$] (w1) at (0,-0.5) {};
					
					\draw (z1) -- (z2);
					\draw (z2) -- (dot1);
					\draw (dot1) -- (za-1) -- (za) -- (dot2)-- (zd);
					\draw (za) -- (wa-1);
					\draw (wa-1) -- (wa-2);
					\draw (wa-2) -- (wdot);
					\draw (wdot) -- (w1);
					\draw (zd) -- (zd+1);
				\end{tikzpicture}
			\end{center}
			\caption{The subgraph induced by the vertex set $\{z_i,w_j: i\in[n], j\in [d+1]\}$}
		\end{figure}
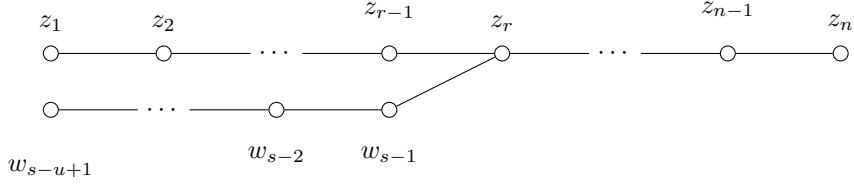
		
		The indices are clearly well-defined. This graph is isomorphic to $L_{n,u}$ where $u\in [3,n-2]$. This contradicts the assumption that $G$ satisfies the $(F_n)$ condition, as desired. This concludes the proof.
	\end{proof}

	\section{Trees whose path ideals have linear quotients}\label{sec:main}

    In this section, we integrate the results established in the preceding sections to establish the main theorem of this article. The objective is to prove a stronger version of Theorem~\ref{thm:main}:
	
    \begin{theorem}\label{thm:main-section-5}
		Let $G$ be a tree and $n\geq 4$ a positive integer. Then the following statements are equivalent:
        \begin{enumerate}
            \item $J_n(G)$ has linear quotients;
            \item $J_n(G)$ has linear resolution;
            \item either of the following holds:
            \begin{itemize}
                \item[(i)] $n=4$, and  $G$ does not contain $P_4+P_4$ or $L_{5,3}$ as an induced subgraph;
                \item[(ii)] $n\geq 5$, and $G$ does not contain $P_n+P_n$ or $L_{n,k}$ as an induced subgraph for  any $k\in [3,\ (n+1)/2]$.
            \end{itemize}
            \item Either $\diam(G)<n-1$ or $G$ satisfies the $(F_n)$ condition.
        \end{enumerate}
	\end{theorem}
    
    The hard part of the proof is the implication $(4)\Rightarrow (1)$. We dedicate most of this section to a proof of this part. By Lemma~\ref{lem:G->trim(G)}, it suffices to consider the case when $G=\trim(G)$, that is, when $G$ is a caterpillar tree.
    
    For the rest of this section, we will use the same notations as given in Notations~\ref{nota:caterpillar} whenever $G$ is a caterpillar tree.     Let $G$ be a caterpillar tree of diameter $d$, and $x_1,\dots, x_{d-1}$ the induced path formed by all vertices of degree $2$ or more of $G$. Recall that, for $1\leq i\leq d-1$, $LN_G(x_i)\eqqcolon\{x_{i1},\dots, x_{il_i}\}$ are called the leaf neighbors of $x_i$,  where $l_1\geq 1, l_{d-1}\geq 1$, and $l_i\geq 0$ if $i\in[2,d-2]$. We impose the following total order on $V(G)$:
		\[
		x_1\succ LN_G(x_1)\succ x_2\succ LN_G(x_2) \succ \cdots \succ  x_{d-1}\succ LN_G(x_{d-1})
		\]
		where $LN_G(x_i)=\{x_{i1}\succ x_{i2}\succ\cdots \succ x_{il_i}\}$. Here by $x\succ \Sigma$ where $\Sigma$ is a set of variables, we mean $x\succ y$ for any $y\in \Sigma$. We now let $(>_{\lex})$ denote the lex total ordering using $(\succ)$ on the set of all monomials in $V(G)$. For each $i\in [1,d-1]$, set $x_{i, \ l_i+1}\coloneqq x_{i+1}$. The idea is that $x_{i+1}$ is also a neighbor of $x_i$, and in our ordering, we have $x_{i1}\succ x_{i2}\succ \cdots \succ x_{i l_i} \succ x_{i, l_i+1}$. 
		
		For a monomial ideal $I$ and a monomial $m$, we denote by $I_{>_{lex}m}$ to be the ideal generated by $m' \in \mingens (I)$ such that $m'>_{lex} m$.

	\begin{remark}
		In our notations, there can be the case of ``dummy" variables. We will treat them as empty set, similar to the common notation that $\{x_1,x_2\dots, x_n\}=\emptyset$ when $n=0$.
	\end{remark}

    We begin with a description of $\mingens(J_n(G))$, where $G$ is a caterpillar tree.
	\begin{lemma}\label{lem:gens-Jn-caterpillar}
		Let $G$ be a caterpillar tree and $n\geq 2$ a positive integer. Then the minimal monomial generators of $J_n(G)$ are of the form $yx_ix_{i+1}\cdots x_{i+n-3}z$ where $i\in [1,\ d-n+2]$, $y\in LN_G(x_i)\cup \{x_{i-1}\}$, and $z\in  LN_G(x_{i+n-3})\cup \{x_{i+n-2}\}$.
	\end{lemma}
	\begin{proof}
		The minimal monomial generators of $J_n(G)$ correspond to paths of length $n-1$ in $G$, and thus, except for at most two vertices, the rest have to be of degree at least 2. In other words, these generators are of the form $yx_ix_{i+1}\cdots x_{i+n-3}z$ where $i\geq 1$, $y\in LN_G(x_i)\cup \{x_{i-1}\}$, and $z\in  LN_G(x_{i+n-3})\cup \{x_{i+n-2}\}$, whenever the notations make sense. Moreover, we also need $i+n-3\leq d-1$, i.e., $i\leq d-n+2$, as desired.
	\end{proof}

    We divide the proof of the implication $(4)\Rightarrow (1)$ in  Theorem~\ref{thm:main-section-5} into three cases, presented in the next three propositions.
	\begin{proposition}\label{prop:LQ-2n-3}
		Let $G$ be a caterpillar tree and $n\geq 4$ a positive integer. Assume that $d=\diam(G)\leq 2n-3$. Then $J_n(G)$ has linear quotients with respect to $(>_{\lex})$.
	\end{proposition}
	\begin{proof}
		Let $m\in \mingens(J_n(G))$. By Lemma~\ref{lem:gens-Jn-caterpillar}, we can set \[m=yx_ix_{i+1}\cdots x_{i+n-3}z\]
		where $i\in [1,\ d-n+2]$, $y\in LN_G(x_i)\cup \{x_{i-1}\}$, and $z\in  LN_G(x_{i+n-3})\cup \{x_{i+n-2}\}$. It now suffices to show that $J_n(G)_{>_{lex}m} \colon m$ is generated by variables, so long as $m$ is not the biggest monomial in $\mingens(J_n(G))$ with respect to $(>_{\lex})$. We have two cases.
		
		\textbf{Case 1:} Assume that $m=x_{i-1}x_ix_{i+1}\cdots x_{i+n-3}x_{i+n-3,v}$ where $i\in [2, d-n+2]$ and $v\in [1, l_{i+n-3}+1]$. For any $y\in LN_G(x_{i-1}) \cup \{x_{i-2}\} \cup \{x_{i+n-3,k}\colon k\in [1,v)\}$, we have $y\succ x_{i+n-3,v}$. In particular, this means that  $m\frac{y}{x_{i+n-3,v}} >_{\lex}  m$, and it is in $\mingens(J_n(G))$ as this monomial corresponds to a path of length $n-1$. Therefore, we have 
		\[
		(LN_G(x_{i-1}) \cup \{x_{i-2}\} \cup \{x_{i+n-3,k}\colon k\in [1,v)\}) \subseteq (J_n(G)_{>_{lex}m}\colon m).
		\]
		We want to show that the converse also holds. Indeed, consider $m'=yx_jx_{j+1}\cdots x_{j+n-3}z>_{\lex}  m$ where $j\in [1,\ d-n+2]$, $y\in LN_G(x_j)\cup \{x_{j-1}\}$, and $z\in  LN_G(x_{j+n-3})\cup \{x_{j+n-2}\})$. It now suffices to show that $(m':m)\subseteq (LN_G(x_{i-1}) \cup \{x_{i-2}\} \cup \{x_{i-n+3,k}\colon k\in [1,v)\})$.
		
		Due to the fact that $m'>_{\lex}  m$ and the structure of $m$, we must have $j<i$, or $j=i, y=x_{j-1}$ and $z\in \{x_{i-n+3,k}\colon k\in [1,v)\}$. Clearly if $j=i$, then $(m'\colon m)\subseteq (z)\subseteq (\{x_{i-n+3,k}\colon k\in [1,v)\})$, and thus the result follows. Note that $z$ here makes sense if and only if $v\geq 2$, and we can assume this as when $v=1$, $m$ is the biggest monomial in $\mingens(J_n(G))$ with respect to $(>_{\lex})$.
		
		We can now assume that $j<i$. If $i=2$ then $j=1$, and we have $m=x_1x_2\cdots x_{n-1}x_{n-1,v}$ (implicitly, this implies that $n\leq d-1$, but we will not need this), and $m'=yx_1x_2\cdots x_{n-2} z$ where $y\in LN_G(x_1)$. Hence $(m'\colon m) \subset (y)\in (LN_G(x_1))=(LN_G(x_{i-1}))$, as desired. Now we can assume that $i\geq 3$, i.e., $x_{i-2}$ is a vertex in $\{x_1,\dots, x_{d-1}\}$. We then have three subcases:
		\begin{itemize}
			\item Assume that $i-2\in [j,j+n-3]$. Then $x_{i-2}\mid m'$, and thus $(m'\colon m) \subseteq (x_{i-2})$, as~desired.
			\item Assume that $i-2<j$. Recall that we have $j<i$. Thus $j=i-1$. Therefore $(m'\colon m)\subseteq (y)\subseteq (LN_G(x_{j})\cup \{x_{j-1}\}) = (LN_G(x_{i-1})\cup \{x_{i-2}\})$, as desired.
			\item Assume that $i-2\geq j+n-3$. Then \[
			j<i-n+1 \leq (d-n+2)-n+1 = d-2n+3 \leq (2n-3)-2n+3 =0,
			\]
			a contradiction.
		\end{itemize}
		
		\textbf{Case 2:} Assume that $m=x_{i,u}x_ix_{i+1}\cdots x_{i+n-3}x_{i+n-3,v}$ where $i\in [1, d-n+2]$, $u\in [1,l_i]$, and $v\in [1, l_{i+n-3}+1]$.    
		For any $y\in \{ x_{i,k} \colon k\in [1,u) \} \cup \{x_{i-1}\} \cup \{x_{i+n-3,k}\colon k\in [1,v)\}$, we have $y\succ x_{i+n-3,v}$. In particular, this means that  $m\frac{y}{x_{i+n-3,v}} >_{\lex} m$, and it is in $\mingens(J_n(G))$ as this monomial corresponds to a path of length $n-1$. Therefore, we have 
		\[
		(\{ x_{i,k} \colon k\in [1,u) \} \cup \{x_{i-1}\} \cup \{x_{i+n-3,k}\colon k\in [1,v)\}) \subseteq (J_n(G)_{>_{lex}m} \colon m).
		\]
		We want to show that the converse also holds. Indeed, consider $m'=yx_jx_{j+1}\cdots x_{j+n-3}z>_{\lex}  m$ where $j\in [1,\ d-n+2]$, $y\in LN_G(x_j)\cup \{x_{j-1}\}$, and $z\in  LN_G(x_{j+n-3})\cup \{x_{j+n-2}\}$. It now suffices to show that $(m':m)\subseteq (\{ x_{i,k} \colon k\in [1,u) \} \cup \{x_{i-1}\} \cup \{x_{i+n-3,k}\colon k\in [1,v)\})$.
		
		Due to the fact that $m'>_{\lex}  m$ and the structure of $m$, one of the following holds:
		\begin{enumerate}
			\item[(a)] $j<i$;
			\item[(b)] $j=i$ and $y\in \{ x_{i,k}\colon k\in [1,u)\}$;
			\item[(c)] $j=i, y=x_{i,u}$, and $z\in \{x_{i-n+3,k}\colon k\in [1,v)\}$. 
		\end{enumerate}
		If (b) holds, then $(m'\colon m)\subseteq (y)\subseteq (\{ x_{i,k}\colon k\in [1,u)\})$, and thus the result follows. Similar arguments apply when (c) holds. We can now assume that (a) holds, i.e., $j\leq i-1$. We have 
		\begin{align*}
			(i-1)-(j+n-3) = i-n+2-j &\leq (d-n+2) -n+2-(1) \\
			&= d-2n+3\\
			&\leq (2n-3)-2n+3\\
			&=0.        
		\end{align*}
		Thus, $i-1\in [j,j+n-3]$. Therefore, $(m'\colon m)\subseteq (x_{i-1}),$ as desired. 
	\end{proof}

	\begin{proposition}\label{prop:LQ-2n-2-alternative}
		Let $G$ be a caterpillar tree and $n\geq 4$ a positive integer. Assume that $d=\diam(G)=2n-2$ and $LN_G(x_{n-2})=\emptyset$, or equivalently, $l_{n-2}=0$. Then $J_n(G)$ has linear quotients with respect to $(>_{\lex})$.
	\end{proposition}
	\begin{proof}
		Let $m\in \mingens(J_n(G))$. By Lemma~\ref{lem:gens-Jn-caterpillar}, we can set \[m=yx_ix_{i+1}\cdots x_{i+n-3}z\]
		where $i\in [1,\ n]$, $y\in LN_G(x_i)\cup \{x_{i-1}\}$, and $z\in  LN_G(x_{i+n-3})\cup \{x_{i+n-2}\}$. It now suffices to show that $J_n(G)_{>_{lex}m} \colon m$ is generated by variables, so long as $m$ is not the biggest monomial in $\mingens(J_n(G))$ with respect to $(>_{\lex})$.
		
		Let $G'$ be the induced subgraph of $G$ with the vertex set $\{x_1,\dots, x_{2n-4},x_{2n-3}\}\sqcup \left( \bigsqcup_{i=1}^{2n-4} LN_G(x_i) \right)$. It is clear that $\mingens(J_n(G'))\subseteq \mingens(J_n(G))$ and the total order $(>_{\lex})$ on $\mingens(J_n(G))$, restricted to $\mingens(J_n(G'))$, is exactly the total order on $\mingens(J_n(G'))$ using Notation~\ref{nota:caterpillar}. We thus will use the same notation. It is straightforward  from our total order $(>_{\lex})$ that
		\[
		\mingens(J_n(G))= \mingens(J_n(G')) \sqcup \Omega
		\]
		and $m_1>_{\lex} m_2$ for any $m_1\in \mingens(J_n(G'))$ and $m_2\in \Omega$, where
		\[
		\Omega\coloneqq \{yx_{n}\cdots x_{2n-3}x_{2n-3,v}\colon y\in LN_G(x_{n})\cup \{x_{n-1}\} \text{ and } v\in [1,l_{2n-3}] \}.\]
		Thus we have $J_n(G)_{>_{lex}m}=J_n(G')_{>_{lex}m}$ as long as $m$ is in $J_n(G')$. Observe that $G'$ is a caterpillar tree with diameter $2n-3$. Hence by Proposition~\ref{prop:LQ-2n-3}, if $m\in J_n(G')$, then the ideal $J_n(G)_{>_{lex}m} \colon m = J_n(G')_{>_{lex}m} \colon m$ is generated by variables. Thus we can now assume that $m\in \Omega,$ i.e., $m= yx_{n}\cdots x_{2n-3}x_{2n-3,v}$
		for some $y\in LN_G(x_{n})\cup \{x_{n-1}\}$ and $v\in [1,l_{2n-3}]$. We have two cases. 
		
		\textbf{Case 1:} Assume that $m= x_{n-1} x_{n}\cdots x_{2n-3}x_{2n-3,v}$
		for some $v\in [1,l_{2n-3}]$. For any $w\in LN_G(x_{n-1}) \cup \{x_{n-2}\} \cup \{x_{2n-3,k}\colon k\in [1,v)\}$, we have $w\succ x_{2n-3,v}$. In particular, this means that  $m\frac{w}{x_{2n-3,v}} >_{\lex}  m$, and it is in $\mingens(J_n(G))$ as this monomial corresponds to a path of length $n-1$. Therefore, we have 
		\[
		(LN_G(x_{n-1}) \cup \{x_{n-2}\} \cup \{x_{2n-3,k}\colon k\in [1,v)\}) \subseteq (J_n(G)_{>_{lex}m} \colon m).
		\]
		We want to show that the converse also holds. Indeed, consider $m'=yx_jx_{j+1}\cdots x_{j+n-3}z>_{\lex}  m$ where $j\in [1,\ d-n+2]$, $y\in LN_G(x_j)\cup \{x_{j-1}\}$, and $z\in  LN_G(x_{j+n-3})\cup \{x_{j+n-2}\}$. It now suffices to show that $(m':m)\subseteq (LN_G(x_{n-1}) \cup \{x_{n-2}\} \cup \{x_{2n-3,k}\colon k\in [1,v)\})$.
		
		Due to the fact that $m'>_{\lex}  m$ and the structure of $m$, we must have $j<n$, or $j=n, y=x_{n-1}$ and $z\in \{x_{2n-3,k}\colon k\in [1,v)\}$. In the latter case, we have $(m'\colon m)\subseteq (z)\subseteq (\{x_{2n-3,k}\colon k\in [1,v)\})$, as desired. We can now assume that $j< n$. First observe that $(n-2)-(j+n-3) = -j+1 \leq 0$, or $n-2\leq j+n-3$. If $n-2\geq j$, then $n-2\in [j,j+n-3]$, and hence  $(m'\colon m)\subseteq (x_{n-2})$, as desired. Now we can assume that $n-2<j<n$, i.e., $j=n-1$. Then   $(m'\colon m)\subseteq (y)\subseteq LN_G(x_{j})=LN_G(x_{n-1})$, as desired.
		
		\textbf{Case 2:} Assume that $m= x_{n,u} x_{n}\cdots x_{2n-3}x_{2n-3,v}$
		for some $u\in[1,l_{n}]$ and $v\in [1,l_{2n-3}]$. For any $w\in \{x_{n,k}\colon k\in [1,u)\} \cup \{x_{n-1}\} \cup \{x_{2n-3,k}\colon k\in [1,v)\}$, we have $w\succ x_{2n-3,v}$. In particular, this means that  $m\frac{w}{x_{2n-3,v}} >_{\lex}  m$, and it is in $\mingens(J_n(G))$ as this monomial corresponds to a path of length $n-1$. Therefore, we have 
		\[
		(\{x_{n,k}\colon k\in [1,u)\} \cup \{x_{n-1}\} \cup \{x_{2n-3,k}\colon k\in [1,v)\}) \subseteq (J_n(G)_{>_{lex}m} \colon m).
		\]
		We want to show that the converse also holds. Indeed, consider $m'=yx_jx_{j+1}\cdots x_{j+n-3}z>_{\lex}  m$ where $j\in [1,\ d-n+2]$, $y\in LN_G(x_j)\cup \{x_{j-1}\}$, and $z\in  LN_G(x_{j+n-3})\cup \{x_{j+n-2}\}$. It now suffices to show that $(m':m)\subseteq (\{x_{n,k}\colon k\in [1,u)\} \cup \{x_{n-1}\} \cup \{x_{2n-3,k}\colon k\in [1,v)\})$. Due to the fact that $m'>_{\lex}  m$ and the structure of $m$, one of the following holds:
		\begin{enumerate}
			\item[(a)] $j<n$;
			\item[(b)] $j=n$ and $y\in \{x_{n,k}\colon k\in [1,u)\}$;
			\item[(c)] $j=n$, $y=x_{n,u}$, and $z\in \{x_{2n-3,k}\colon k\in [1,v)\}$.
		\end{enumerate}
		If (b) holds, then $(m'\colon m)\subseteq (y) \subseteq (\{x_{n,k}\colon k\in [1,u)\})$, as desired. Similar arguments apply when (c) holds. We can now assume that (a) holds, i.e., $j\leq n-1$. If $n-1\in [j,j+n-3]$, then $(m'\colon m) \subseteq (x_{n-1})$, as desired. Now we can assume that $n-1>j+n-3$, or equivalently, $j<2$. Therefore we have $j=1$. We then have $z\in LN_G(x_{j+n-3})\cup \{x_{j+n-2}\} = LN_G(x_{n-2})\cup \{x_{n-1}\} = \{x_{n-1}\}$ due to the hypotheses. Thus $(m'\colon m) \subseteq (x_{n-1})$, as desired.
	\end{proof}
	
	\begin{proposition}\label{prop:LQ-2n-1}
		Let $G$ be a caterpillar tree and $n\geq 4$ a positive integer. Assume that $d=\diam(G)=2n-1$ and $LN_G(x_{n-2})=LN_G(x_{n+1})=\emptyset$, or equivalently, $l_{n-2}=l_{n+1}=0$. Then $J_n(G)$ has linear quotients with respect to $(>_{\lex})$.
	\end{proposition}
	\begin{proof}
		Let $m\in \mingens(J_n(G))$. By Lemma~\ref{lem:gens-Jn-caterpillar}, we can set \[m=yx_ix_{i+1}\cdots x_{i+n-3}z\]
		where $i\in [1,\ n+1]$, $y\in LN_G(x_i)\cup \{x_{i-1}\}$, and $z\in  LN_G(x_{i+n-3})\cup \{x_{i+n-2}\}$. It now suffices to show that $J_n(G)_{>_{lex}m} \ \colon m$ is generated by variables, so long as $m$ is not the biggest monomial in $\mingens(J_n(G))$ with respect to $(>_{\lex})$.
		
		Let $G'$ be the induced subgraph of $G$ with the vertex set $\{x_1,\dots, x_{2n-3},x_{2n-2}\}\sqcup \left( \bigsqcup_{i=1}^{2n-3} LN_G(x_i) \right)$. It is clear that $\mingens(J_n(G'))\subseteq \mingens(J_n(G))$ and the total order $(>_{\lex})$ on $\mingens(J_n(G))$, restricted to $\mingens(J_n(G'))$, is exactly the total order on $\mingens(J_n(G'))$ using Notation~\ref{nota:caterpillar}. We thus will use the same notation. It is straightforward from our total order $(>_{\lex})$ that
		\[
		\mingens(J_n(G))= \mingens(J_n(G')) \sqcup \Omega
		\]
		and $m_1>_{\lex} m_2$ for any $m_1\in \mingens(J_n(G'))$ and $m_2\in \Omega$, where
		\[
		\Omega\coloneqq \{yx_{n+1}\cdots x_{2n-2}x_{2n-2,v}\colon y\in LN_G(x_{n+1})\cup \{x_{n}\} \text{ and } v\in [1,l_{2n-2}] \}.\]
		Thus we have $J_n(G)_{>_{lex}m} =J_n(G')_{>_{lex}m} $ as long as $m$ is in $J_n(G')$. Observe that $G'$ is a caterpillar tree with diameter $2n-2$, with $LN_{G'}(x_{n-2})=\emptyset$. Hence by Proposition~\ref{prop:LQ-2n-2-alternative}, if $m\in J_n(G')$, then the ideal $J_n(G)_{>_{lex}m} \ \colon m = J_n(G')_{>_{lex}m} \ \colon m$ is generated by variables. Thus we can now assume that $m\in \Omega,$ i.e., $m= yx_{n+1}\cdots x_{2n-2}x_{2n-2,v}$
		for some $y\in LN_G(x_{n+1})\cup \{x_{n}\}$ and $v\in [1,l_{2n-2}]$. Recall from our hypotheses that $LN_G(x_{n+1})=\emptyset$. Therefore, we have     
		\[
		m= x_{n}x_{n+1}\cdots x_{2n-2}x_{2n-2,v}
		\]
		for some $v\in [1,l_{2n-2}]$. 
		
		Observe that for any $w\in LN_G(x_{n}) \cup \{x_{n-1}\} \cup \{x_{2n-2,k}\colon k\in [1,v)\}$, we have $w\succ x_{2n-2,v}$. In particular, this means that  $m\frac{w}{x_{2n-2,v}} >_{\lex}  m$, and it is in $\mingens(J_n(G))$ as this monomial corresponds to a path of length $n-1$. Therefore, we have 
		\[
		(LN_G(x_{n}) \cup \{x_{n-1}\} \cup \{x_{2n-2,k}\colon k\in [1,v)\}) \subseteq (J_n(G)_{>_{lex}m} \ \colon m).
		\]
		We want to show that the converse also holds. Indeed, consider $m'=yx_jx_{j+1}\cdots x_{j+n-3}z>_{\lex}  m$ where $j\in [1,\ n+1]$, $y\in LN_G(x_j)\cup \{x_{j-1}\}$, and $z\in  LN_G(x_{j+n-3})\cup \{x_{j+n-2}\}$. It now suffices to show that $(m':m)\subseteq (LN_G(x_{n}) \cup \{x_{n-1}\} \cup \{x_{2n-2,k}\colon k\in [1,v)\})$.
		
		Due to the fact that $m'>_{\lex}  m$ and the structure of $m$, we must have $j<n+1$, or $j=n+1, y=x_{n}$ and $z\in \{x_{2n-2,k}\colon k\in [1,v)\}$. In the latter case, we have $(m'\colon m)\subseteq (z)\subseteq (\{x_{2n-2,k}\colon k\in [1,v)\})$, as desired. We can now assume that $j< n+1$. We have three subcases:
		\begin{itemize}
			\item Assume that $n-1\in [j,j+n-3]$. Then $(m'\colon m)\subseteq (x_{n-1})$, as desired.
			\item Assume that $n-1<j$. Since $j<n+1$, we must have $j=n$. Then  $(m'\colon m) \subseteq (y) \subseteq (LN_G(x_{n})\cup \{x_{n-1}\})$, as desired.
			\item Assume that $n-1>j+n-3$, or equivalently, that $j<2$. We then must have $j=1$. We then have $z\in LN_G(x_{n-2})\cup \{x_{n-1}\}= \{x_{n-1}\}$ since we know that $LN_G(x_{n-2})=\emptyset$ from the hypotheses. Thus $(m'\colon m)\subseteq (x_{n-1})$, as desired. \qedhere
		\end{itemize}
	\end{proof}

	\begin{proof}[Proof of Theorem~\ref{thm:main-section-5}]
		$(1)\Rightarrow (2)$ is Lemma~\ref{lem:LQ-implies-LR}, $(2)\Rightarrow (3)$ follows from Lemmas~\ref{lem:linear-resolution-induced}, \ref{lem:forbidden-Pn+Pn}, \ref{lem:forbidden-structures}, and \ref{lem:forbidden-for-J4}, and $(3)\Rightarrow (4)$ is Observation~\ref{obser}. 
        
        We now show $(4)\Rightarrow (1)$. If $\diam(G)<n-1$, then $J_n(G)=(0)$ has linear quotients trivially. Now we can assume that $G$ satisfies the $(F_n)$ condition. By Lemma~\ref{lem:G->trim(G)}, we can assume that $G$ is a caterpillar tree. Note that due to the $(F_n)$ condition, we have $\diam(G)\leq 2n-1$. 
        
        If $\diam(G)\leq 2n-3$, then $J_n(G)$ has linear quotients by Proposition~\ref{prop:LQ-2n-3}.
		
		If $\diam(G)=2n-2$, then we claim that  either $LN_G(x_{n-2})$ or $LN_G(x_{n})$ is empty. Suppose otherwise that $l_{n-2},l_{n}\geq 1$. Then the induced subgraph of $G$ with the vertex set \[
        \{x_{1,1}, x_{1},x_2, \dots, x_{n-2}, x_{n-2,1}\}\cup \{x_{n,1},x_{n},\dots, x_{2n-3},x_{2n-3,1}\}\]
        is $P_n+P_n$, a contradiction. Thus the claim holds. Without loss of generality, assume that $LN_G(x_{n-2})=\emptyset$. Then $J_n(G)$ has linear quotients by Proposition~\ref{prop:LQ-2n-2-alternative}.
		
		Finally, we can assume that $\diam(G)=2n-1$. We claim that $LN_G(x_{n-2})=LN_G(x_{n+1})=0$. Suppose otherwise that one of them is non-empty. Without loss of generality, we can assume that $LN_G(x_{n-2})\neq \emptyset$, i.e., $l_{n-2}\geq 1$. Then the induced subgraph of $G$ with the vertex set $\{x_{1,1},x_1, x_2, \dots, x_{n-2}, x_{n-2,1}\}\cup \{x_{n},x_{n+1},\dots, x_{2n-3},x_{2n-2},x_{2n-2,1}\}$ is $P_n+P_n$, a contradiction. Thus the claim holds. Then $J_n(G)$ has linear quotients by Proposition~\ref{prop:LQ-2n-1}. This concludes the~proof.
	\end{proof}

	\section{Final remarks}
	
	The most important technique in this article is arguably the trimming operation. As briefly mentioned in the introduction, adding edges to a graph does not necessarily create more generators for its path ideal. In other words, for a fixed integer $n\geq 2$ and a graph $G$, there exists an induced subgraph $H$ of $G$ such that $J_n(H)=J_n(G)$. We note that when $n=2$, we must have $H=G$. In this article, when $G$ is a tree such that $J_n(G)$ has linear resolution, such an $H$ is a caterpillar graph, which drastically simplifies the problem. One can ask whether one can generalize Lemma~\ref{lem:G->trim(G)}, and in turn Theorem~\ref{thm:main}, or find analogs for other interesting classes of ~graphs.

	\bibliographystyle{amsplain}
	\bibliography{refs}
	
\end{document}